\documentclass{amsart}
\usepackage[latin1]{inputenc}
\usepackage[active]{srcltx}
\usepackage{verbatim}
\usepackage{hyperref}
\usepackage{mathrsfs}

\usepackage{latexsym,amsmath,amsxtra,amssymb,mathrsfs,enumerate, color, graphicx}
\usepackage[T1]{fontenc}

\definecolor{orange}{RGB}{255,127,0}

\usepackage[fleqn,tbtags]{mathtools}
\mathtoolsset{showonlyrefs,showmanualtags}

\theoremstyle{plain}
\newtheorem{theorem}{Theorem}[section]        
     
\newtheorem{lemma}[theorem]{Lemma}             
\newtheorem{proposition}[theorem]{Proposition} 
\theoremstyle{remark}
     
\newtheorem{example}[theorem]{Example}
\newtheorem{definition}[theorem]{Definition}   
               
\newtheorem{remark}[theorem]{Remark}        

\numberwithin{equation}{section}

\def\R{{\mathbb R}}
\def\C{{\mathbb C}}

\newcommand{\E}{{\mathbb E}}
\renewcommand{\P}{{\mathbb P}}
\newcommand{\F}{{\mathscr F}}

\renewcommand{\H}{E}
\renewcommand{\O}{\Omega}

\newcommand{\beq}{\begin{equation}}
\newcommand{\eeq}{\end{equation}}
\newcommand{\bal}{\begin{aligned}}
\newcommand{\eal}{\end{aligned}}
\newcommand{\ben}{\begin{enumerate}}
\newcommand{\een}{\end{enumerate}}
\newcommand{\bit}{\begin{itemize}}
\newcommand{\eit}{\end{itemize}}
\newcommand{\bth}{\begin{theorem}}
\renewcommand{\eth}{\end{theorem}}
\newcommand{\bpr}{\begin{proposition}}
\newcommand{\epr}{\end{proposition}}
\newcommand{\ble}{\begin{lemma}}
\newcommand{\ele}{\end{lemma}}
\newcommand{\bpf}{\begin{proof}}
\newcommand{\epf}{\end{proof}}
\newcommand{\bex}{\begin{example}}
\newcommand{\eex}{\end{example}}
\newcommand{\bre}{\begin{example}}
\newcommand{\ere}{\end{example}}

\renewcommand{\Im}{\hbox{\rm Im}\,}
\renewcommand{\Re}{\hbox{\rm Re}\,}

\newcommand{\calL}{{\mathscr L}}
\newcommand{\n}{\Vert}

\newcommand{\one}{{{\bf 1}}}
\newcommand{\embed}{\hookrightarrow}

\renewcommand{\a}{\alpha}
\renewcommand{\b}{\beta}

\newcommand{\Dom}{\mathsf{D}}

\newcommand{\fint}[0]{-\!\!\!\!\!\!\int}

\newcommand{\W}{W} 

\newcommand{\tmbp}{T_{2,\beta}^{p,2}}
\newcommand{\tmbpH}{T_{2,\beta}^{p,2}(H)}
\newcommand{\tmb}{T_{2,\beta}^{2,2}}
\newcommand{\tmbH}{T_{2,\beta}^{2,2}(H)}
\newcommand{\Tmb}{T_{2,\beta}^{1,2}}
\newcommand{\TmbH}{T_{2,\beta}^{1,2}(H)}
\newcommand{\tbp}{T_{2,\beta}^{p,2}}
\newcommand{\tbpH}{T_{2,\beta}^{p,2}(H)}
\newcommand{\tb}{T_{2,\beta}^{2,2}}
\newcommand{\tbH}{T_{2,\beta}^{2,2}(H)}

\allowdisplaybreaks

\title[Conical stochastic maximal $L^p$-regularity]
{Conical stochastic maximal $L^p$-regularity \\ for $1\leq p<\infty$}

\author{Pascal Auscher} 
\address{Universit\'e \ Paris-Sud, \ Laboratoire \  de \  Math\'ematiques, \ UMR \ 8628, \ F-91405 \ {\sc Orsay}; \ CNRS, F-91405 {\sc Orsay}.} 
\email{pascal.auscher@math.u-psud.fr}

\author{Jan van Neerven}
\address{Delft Institute of Applied Mathematics\\
Delft University of Technology\\P.O. Box 5031\\2600 GA Delft\\The Netherlands}
\email{J.M.A.M.vanNeerven@tudelft.nl}

\thanks{Jan van Neerven is supported by VICI subsidy 639.033.604 
of the Netherlands Organisation for Scientific Research (NWO)}

\author{Pierre Portal}
\address{Permanent Address:
Universit\'e Lille 1, Laboratoire Paul Painlev\'e, F-59655 {\sc Villeneuve d'Ascq}.
Current Address:
Australian National University, Mathematical Sciences Institute, John Dedman Building, 
Acton ACT 0200, Australia.}
\email{Pierre.Portal@anu.edu.au}

\keywords{Conical maximal $L^p$-regularity, stochastic convolutions, tent spaces, off-diagonal estimates}
\subjclass[2000]{Primary: 60H15; Secondary: 42B25, 42B35, 47D06}
\date\today

\begin{document}

\begin{abstract} Let $A = -{\rm div} \,a(\cdot) \nabla$ be a second order divergence form elliptic operator on $\R^n$ with bounded 
measurable real-valued coefficients and let $W$ be a cylindrical Brownian motion in a Hilbert space $H$.
Our main result implies
that the stochastic convolution process
$$ u(t) = \int_0^t e^{-(t-s)A}g(s)\,dW(s), \quad t\ge 0,$$
satisfies, for all $1\le p<\infty$,
a conical maximal $L^p$-regularity estimate
$$\E \n \nabla u \n_{ T_2^{p,2}(\R_+\times\R^n)}^p 
\le C_p^p \E \n g \n_{ T_2^{p,2}(\R_+\times\R^n;H)}^p.$$
Here, $T_2^{p,2}(\R_+\times\R^n)$ and $T_2^{p,2}(\R_+\times\R^n;H)$ are 
the parabolic tent spaces of real-valued and $H$-valued functions, respectively. 
This contrasts with Krylov's maximal $L^p$-regularity estimate  
$$\E \n \nabla u \n_{L^p(\R_+;L^2(\R^n;\R^n))}^p 
\le C^p \E \n g \n_{L^p(\R_+;L^2(\R^n;H))}^p$$
which is known to hold only for $2\le p<\infty$, even when $A = -\Delta$ and $H = \R$.

The proof is based on an $L^2$-estimate and extrapolation arguments which use the fact that 
$A$ satisfies suitable off-diagonal bounds. 
Our results are applied to obtain conical stochastic maximal $L^p$-regularity for 
a class of nonlinear SPDEs with rough initial data.
\end{abstract}

\maketitle

\section{Introduction}
Let us consider the following stochastic heat equation in $\R^n$
driven by a cylindrical Brownian motion $W$ with values in a (finite- or infinite-dimensional) Hilbert space $H$:
\begin{equation}\label{eq:SDE} \left\{
\begin{aligned}
\frac{\partial u}{\partial t}(t,x) & = \Delta u(t,x)+ g(t,x)\dot W(t), && t\ge 0, \ x\in \R^n, \\
u(0,x) & = 0, && x\in\R^n.
\end{aligned}
\right.
\end{equation}
Under suitable measurability and integrability conditions on the process 
$g:\R_+\times\R^n\times\Omega\to H$,
the process $u:\R_+\times\R^n\times\Omega\to \R$ given formally by 
the stochastic convolution 
\begin{align}\label{eq:stoch-convol} 
u(t) = \int_0^t e^{(t-s)\Delta} g(s)\,dW(s), 
\quad t\ge 0,
\end{align}
is well defined. This process is usually called the mild solution of \eqref{eq:SDE},
and it has stochastic maximal $L^2$-regularity in the sense that
$$\E\n \nabla u \n_{ L^2(\R_+; L^2(\R^n;\R^n))}^2 \le C^2\E\n g\n _{L^2(\R_+; L^2(\R^n;H))}^2$$
with a constant $C$ independent of $g$ and $H$.
This follows from a classical result of Da Prato \cite{DP} (see \cite{DPZ}
for further results along these lines). It was 
subsequently shown by Krylov \cite{Kry94, Kry06} that, 
for $p\ge 2$, $u$ has stochastic maximal $L^p$-regularity 
in the sense that 
$$\E\n \nabla u \n_{ L^p(\R_+; L^2(\R^n;\R^n))}^p \le C_p^p \E\n g\n _{L^p(\R_+; L^2(\R^n;H))}^p.$$
Krylov actually proves that
$L^2(\R^n;H)$ may be replaced by $L^q(\R^n;H)$ for any $2\le q\le p$
and that $\Delta$ may be replaced by any second-order uniformly elliptic operator 
under mild regularity assumptions on the coefficients.
For $p>2$, the condition $q\le p$ was removed in \cite{NVW11}
and the result was extended to arbitrary operators having a bounded $H^\infty$-calculus on $L^q(X,\mu)$,
where $q\ge 2$ and $(X,\mu)$ is an arbitrary $\sigma$-finite 
measure space. 

The condition $p\ge 2$ in all these results is necessary, 
in the sense that the corresponding result for 
$1\le p<2$ is false  even for $H=\R$ \cite{Kry}.
The aim of this paper is to show that the stochastic heat equation \eqref{eq:SDE} 
does have `conical' stochastic maximal $L^p$-regularity in 
the full range of $1\leq p<\infty$, provided the condition 
$g\in L^p(\R_+\times\Omega; L^2(\R^n;H))$ is replaced by the condition $g\in 
L^p(\Omega; T_2^{p,2}(\R_+\times \R^n, t^{-\beta}dt\times dx;H))$. Here $T_2^{p,2}(\R_+\times \R^n, 
t^{-\beta}dt\times dx;H)$ is a
weighted parabolic tent space of $H$-valued functions on $\R_+\times \R^n$ 
(the definition is stated in Section \ref{sec:main}; for $\beta=1$
the classical parabolic tent space $T_2^{p,2}(\R_+\times \R^n;H)$ is obtained). 
Our main result, stated somewhat informally (see Theorem \ref{thm:main} for the precise formulation),
reads as follows.

\begin{theorem}\label{thm:stoch-heat}
Let $A = -{\rm div}\,a(\cdot)\nabla$ be a divergence form elliptic operator on $\R^n$ 
with bounded measurable real-valued coefficients. 
Then for all $1\leq p<\infty$ and $\beta>0$  
the stochastic convolution process
$$ u(t) = \int_0^t e^{-(t-s)A}g(s)\,dW(s), \quad t\ge 0,$$
satisfies the conical stochastic maximal $L^p$-regularity estimate
 $$\E \n \nabla u \n_{ T_2^{p,2}(\R_+\times \R^n, t^{-\beta}dt\times dx;\R^n)}^p 
\le C_{p,\beta}^p \E \n g \n_{ T_2^{p,2}(\R_+\times \R^n, t^{-\beta}dt\times dx;H)}^p.$$
\end{theorem}

The precise assumptions on $A$ are stated in Example \ref{ex:divform} below.
The proof of Theorem \ref{thm:stoch-heat} proceeds in two steps. First, a $T_2^{2,2}$-estimate is 
deduced from the It\^o isometry (Section \ref{sec:T22}). Using off-diagonal bound techniques, this estimate
is then extrapolated to a $T_2^{p,2}$-estimate (Section \ref{sec:extrapolation}). 

The results are applied to prove conical maximal $L^p$-regularity 
for a class of stochastic partial 
differential equations on $\R^n$ driven by space-time white noise (Section \ref{sec:SPDE}). 
We shall prove that if 
$b:\R^n\to \R$ satisfies appropriate Lipschitz and growth assumptions and $A$ is as in Theorem \ref{thm:stoch-heat}, 
then the mild solution of the stochastic PDE
\begin{equation*} \left\{
\begin{aligned}
\frac{\partial u}{\partial t}(t,x) + A u(t,x) & =  b(\nabla u(t,x))\dot W(t), && t\ge 0, \ x\in \R^n, \\
u(0,x) & = u_0(x), && x\in\R^n,
\end{aligned}
\right.
\end{equation*}
has conical stochastic maximal $L^p$-regularity for all $1<p<\infty$, in the sense that 
$\nabla u\in T_2^{p,2}(\R_+\times \R^n,t^{-\beta}dt\times dx)$ for all 
$0<\beta <1$ 
and all initial values $u_0\in \Dom_p(A^{\frac{\beta}{2}})$,
the domain of the $L^p$-realisation of 
$A^{\frac{\beta}{2}}$. Note that 
the weight $t^{-\beta}$ allows the handling
of initial values in $\Dom_p(A^\theta)$ with $\theta>0$ arbitrarily small. 
It is only the stochastic part that forces us to take 
$\beta>0$, and it seems that our technique does not work when $\beta=0$.

The present paper, as well as \cite{AKMP} which contains more elaborate 
developments not needed here, builds upon techniques developed in \cite{AMP}.
There, similar off-diagonal bound techniques are applied
to obtain conical maximal $L^p$-regularity for a class of deterministic initial value problems.
The key feature of both papers is that they depart from the traditional paradigm in the theory of evolution 
equations where a solution is a trajectory, indexed by time, in a suitably chosen state space.
This could be called the `Newtonian' paradigm, in which time and space are treated as separate entities.
In the conical approach, space and time are inextricably mixed into one `space-time'.

The idea of using tent space maximal regularity in PDEs goes back, as far as we know, to Koch and Tataru \cite{kt}, 
who proved $T^{\infty,2}$-regularity of solutions of Navier-Stokes equations with rough initial data (see also \cite{kl}).
The underlying ideas come from the theory of Hardy spaces and its application to boundary value problems (see, e.g. \cite{dk}). 
To the best of our knowledge, the present paper is the first to consider a tent space approach for stochastic PDEs. 

The notations in this paper are standard. For unexplained terminology we refer to 
\cite{NVW11b, nvw, nw} (concerning cylindrical Brownian motions and vector-valued stochastic integration) and
\cite{Stein} (concerning tent spaces). We use the convention $\R_+ = (0,\infty)$. We work over the real scalar field.

\section{Preliminaries}\label{sec:main}

\subsection{Off-diagonal bounds}
Our results rely on off-diagonal bound techniques. A family $(T_t)_{t>0}$ of bounded linear operators
on $L^2(\R^n)$ 
is said to satisfy {\em $L^q$-$L^2$ off-diagonal bounds}   
if there exist constants $c>0$ and $C\ge 0$ such that for all Borel sets 
$E$, $F$ in $\R^n$ and all $f\in L^2\cap L^q(\R^n)$ we have $$ \n \one_E T_t \one_F f\n_{L^2(\R^n)} \le C t^{-\frac{n}{2}(\frac{1}{q}-\frac{1}{2})}
\exp(-c(d(E,F))^2/t)  \n \one_F f\n_{L^q(\R^n)},$$  
with $d(E,F) := \inf\{|x-y|: \ x\in E, \ y\in F\}$.

\medskip

Such bounds are substitutes for the classical pointwise kernel estimates of Calder\-\'on-Zygmund theory, which are 
not available when one deals with semigroups generated by elliptic operators with rough coefficients. 
Following the breakthrough paper \cite{bk}, they have recently become a highly popular tool in harmonic analysis.
Typical examples of their use are given in the memoir \cite{A}.  Note that 
$L^2$-$L^2$ off-diagonal bounds imply uniform boundedness in $L^{2}$ (taking $E=F=\R^{n}$).
Observe that off-diagonal bounds form an ordered scale of conditions. 

 \begin{lemma}\label{lem:comparisonoff} Let $1\le q\le r\le 2$, and  $(T_{t})_{t>0}$ be a family of bounded linear operators on $L^2$, which satisfies $L^q$-$L^2$ off-diagonal bounds. 
Then $(T_{t})_{t>0}$ satisfies $L^r$-$L^2$ off-diagonal bounds.
\end{lemma}
\begin{proof}
This is a consequence of \cite[Proposition 3.2]{am}, where it is proven that such off-diagonal bounds are equivalent, on $\R^{n}$, to off-diagonal bounds on balls (see \cite[Definition 2.1]{am}). 
The result for the latter follows from H\"older's inequality.
\end{proof}

\begin{example}[Divergence form elliptic operators]
\label{ex:divform}
We mostly consider second order operators in divergence form 
$A= -{\rm div}\, a \nabla$, with $a \in L^{\infty}(\R^{n};M^n(\R))$ elliptic in the sense that there exist
 $C,C'>0$ such that for all $x\in\R^n$ and $\xi,\xi' \in \R^{n}$ we have 
$$
 a(x)\xi\cdot \xi \ge C|\xi|^{2} \quad  \text{and} \quad |a(x)\xi\cdot\xi'|\leq C' |\xi||\xi'|.
$$ 
It is proven in \cite[subsection 4.3]{A} that $(t^{\frac12}\nabla e^{-tA})_{t \geq 0}$ satisfies 
$L^{q}$-$L^{2}$ off-diagonal bounds for all $q \in (1,2]$.
In fact, $(t^{\frac12}\nabla e^{-tA})_{t \geq 0}$ even satisfies $L^{1}$-$L^{2}$ off-diagonal bounds, as can be seen
in \cite[page 51]{at} as a consequence of \cite[Theorem 4 and Lemma 20]{at}.
We use these $L^{1}$-$L^{2}$ bounds in the results below. 
If we assume only $L^{q}$-$L^{2}$ bounds for some $q\in (1,2]$, Theorem \ref{thm:abstract-q} 
still holds for all $p\in [1,2] \cap (\frac{2n}{n+\beta q'},2]$  (where $\frac1q+\frac1{q'}=1$),   
but the proof is technical (see \cite{AKMP}). This version suffices for proving 
Theorem \ref{thm:stoch-heat}.
 
Note that we assume that $a$ has real-valued coefficients.
In the stochastic setting, where the noise process $W$ is also real-valued, this is a natural
assumption. 
\end{example}

\subsection{Conical maximal $L^p$-regularity}
\label{subsec:conical}
The notion of maximal $L^p$-regularity has played an important role in
much of the recent progress in the theory of nonlinear parabolic evolution 
equations. We refer to the lecture notes of Kunstmann and Weis \cite{KW}
for an overview and references to the rapidly expanding literature on this
topic. 

Motivated by applications to 
boundary value problems with $L^2$-data,
Auscher and Axelsson \cite{AA-Inventiones,AA} proved that for a bounded 
analytic $C_0$-semigroups $S = (S(t))_{t\ge 0}$ with generator $-A$
on a Hilbert space $\H$, the classical maximal $L^2$-regularity estimate
\begin{align}\label{eq:MR} \n AS * g\n_{L^2(\R_+;\H)}
\le C\n g\n_{L^2(\R_+;\H)}
\end{align}
implies, for any $\beta\in (-1,\infty)$, the weighted 
maximal $L^2$-regularity estimate 
\begin{align}\label{eq:MR-beta}\n AS * g\n_{L^2(\R_+, t^{-\beta}dt; \H)}
\le C_\beta\n g\n_{L^2(\R_+, t^{-\beta}dt;\H)}.
\end{align}
Here,
$$ S*g(t) = \int_0^t S(t-s)g(s)\,ds$$
denotes the convolution of $g$ with the semigroup $S$ and $AS * g := A(S*g)$. 
See also \cite{ps} for similar weighted maximal regularity estimates in $L^{p}$ spaces. 
 
With the aim of eventually extending the results of \cite{AA-Inventiones} to 
an $L^p$-setting, a `conical' $L^p$-version of  \eqref{eq:MR-beta}
was subsequently obtained in \cite{AMP}. 
Observing that, for $\H = L^2(\R^n)$, one has
\begin{equation}\label{eq:T22} \n g\n_{L^{2}(\R_+,t^{-\beta}dt;L^2(\R^n))}
=\Big( \int  _{\R^{n}} 
\Big(\int  _{0} ^{\infty}\!\! \fint_{B(x,t^\frac12)} | g(t,y)|^{2} \, dy\,\frac{dt}{t^{\beta}} \Big) \,dx \Big)^{\frac{1}{2}},
\end{equation}
where the dashed integral denotes the average over the ball 
$B(x,t^\frac12) = \{y\in \R^n: |x-y|<t^\frac12\}$.
One defines,
for $1\le p<\infty$,
$$
\|g\|_{T_2^{p,2}(\R_+\times\R^n, t^{-\beta}dt\times dy)} := \Big( \int  _{\R^{n}} 
\Big(\int  _{0} ^{\infty}\!\! \fint _{B(x,t^\frac12)}| g(t,y)|^{2} 
\,dy\,\frac{dt}{t^{\beta}} \Big)^{\frac{p}{2}} \,dx \Big)^{\frac{1}{p}}.
$$
The Banach space 
$$\tmbp  := T_2^{p,2}(\R_+\times\R^n;t^{-\beta}dt\times dy)$$ 
consisting of all measurable
functions $g:\R_+\times\R^n\to \R$ 
for which this norm is finite is called the {\em tent space} of exponent $p$ and weight $\beta$.
The spaces $\tmbp $ are weighted, parabolic versions of the spaces $T^{p,2}$ introduced by Coifman, Meyer and Stein \cite{CMS}, and have been\
studied by many authors. We refer to \cite{Stein} for a thorough discussion and
references to the literature. It is useful to observe that 
$\tmbp $ can be identified with a closed 
subspace of  $L^{p}(\R^n;L^{2}(t^{-\frac{n}{2}-\beta}dt\times dy))$ 
for $1<p<\infty$ and of the Hardy space 
$H^{1}(\R^n;L^{2}(t^{-\frac{n}{2}-\beta} dt\times dy))$ for $p=1$
(see \cite{HTV}).

\medskip
\noindent{\em Notation.} \ 
From now on, whenever functions belong to a (vector-valued) 
Lebesgue space over $\R^n$, we shall suppress $\R^n$ 
from our notations. For instance, we shall write 
$$L^2 := L^2(\R^n), \quad L^2(H) = L^2(\R^n;H)$$ 
and thus use the notation $L^2(\R^n)$ as an abbreviation for $L^2(\R^n;\R^n)$.
Likewise we suppress $\R_+\times \R^n$ from the notations for (vector-valued) tent spaces. In all other instances
we shall be notationally more explicit.

\medskip
The next estimate is the main result of \cite{AMP}.

\begin{theorem}[Conical maximal $L^p$-regularity]\label{thm:conicalmaxreg} Let $-A$ be the generator of a bounded
analytic $C_0$-semigroup $S = (S(t))_{t \geq 0}$ on $L^2$,
and suppose that the family $(tAS(t))_{t \geq 0}$ 
satisfies $L^2$-$L^2$ off-diagonal bounds.
Then for all $\beta>-1$, $p> \sup \big(\frac{2n}{n+ 2(1+\beta)}, 1\big)$,  and 
$g\in L^2(t^{-\b}\,dt;\Dom(A))\cap \tmbp$
one has
$$ \n AS*g\n_{\tmbp } \le C_{p,\beta} \n g\n_{\tmbp },$$
with constant $C_{p,\beta}$ independent of $g$.
\end{theorem}

It is routine to see that the inclusions
$$L^2(t^{-\b}\,dt;\Dom(A))\cap \tmbp \embed L^2(t^{-\b}\,dt;\R^n)\cap \tmbp
= \tmb \cap \tmbp \embed \tmbp$$ are dense, so the above result gives the unique
extendability of $g\mapsto AS*g$ to a bounded operator on $\tmbp$.

The proof of this result, as well as that of Theorem \ref{thm:main} below,
depends on a change of aperture result for tent spaces. 
Tent spaces with 
{\em aperture $\alpha > 0 $} are defined by the norms
$$
\|g\|_{T_{2,\beta,\alpha}^{p,2}} := \Big( \int  _{\R^{n}} 
\Big(\int  _{0} ^{\infty} \fint  _{B(x,\alpha t^\frac12)}| g(y,t)|^{2} \,dy\,
\frac{dt}{t^{\beta}} \Big)^{\frac{p}{2}} \,dx \Big)^{\frac{1}{p}}.
$$
For all $\alpha \ge 1$  one has
\begin{align}\label{eq:hnp}  \|g\|_{T_{2,\beta,\alpha}^{p,2}}
\le C \a^{n/(p\wedge 2)} \|g\|_{\tmbp }
\end{align}
for some constant $C$ independent of $\alpha$ and $m$. This was first
proved in \cite{HNP} in a vector-valued context, but with an additional logarithmic factor. 
A different proof in the scalar-valued case was obtained in \cite{a-angle}. The important point is that the right-hand side improves the classical
bound from \cite{CMS}. The weighted parabolic situation treated here follows from these results applied to the function $(t,y)\mapsto t^{\beta+1}f(t^2,y)$ (see \cite{AMP}). 
For later use we mention that the bounds \eqref{eq:hnp} extend 
to the Hilbert space-valued tent spaces $\tmbpH$ (which are defined 
in the obvious way).

\section{The main result}

Given a probability space $(\Omega,\F,\P)$ 
endowed with a filtration $\F = (\F_t)_{t\ge 0}$ and a real Hilbert space $H$, unless stated otherwise, 
$\W = (\W(s))_{s \geq 0}$ denotes a $\mathscr{F}$-cylindrical Brownian 
motion in $H$ (see, e.g., \cite{NVW11b} for the precise definition)
which we consider to be fixed throughout the rest of the paper.
In applications to stochastic partial differential
equations, one typically takes $H$ to be $L^2(D)$
for some domain $D\subseteq \R^n$; this provides the mathematically
rigorous model for space-time white noise on $D$. Also note that for $H=\R^d$,
$\W$ is just a standard $\F$-Brownian motion in $\R^d$.

An {\em $\F$-adapted simple process} with values in $H$ is a measurable mapping 
$g:\R_+\times\R^n\times\Omega\to H$ of the form
$$ g(t,x,\omega) = \sum_{\ell=1}^N \one_{(t_\ell, t_{\ell+1}]}(t)
\sum_{m=1}^N \one_{A_{m\ell}}(\omega) \phi_{m\ell}(x)$$
with $0\le t_1<\dots<t_N<t_{N+1}<\infty$, $A_{m\ell}\in\mathscr{F}_{t_\ell}$, 
and $\phi_{m\ell}$ simple functions on $\R^n$ 
with values in $H$. For such processes, the stochastic convolution process
$$
S\diamond g(t):= \int_{0} ^{t} S(t-s)g(s)\,d\W(s)
$$
is well-defined as an $L^2$-valued process 
whenever $S = (S(t))_{t\ge 0}$ is a $C_0$-semigroup of bounded linear operators on $L^2$
(see, e.g., \cite{nvw}). 

The main result of this paper reads as follows.

\begin{theorem}[Conical stochastic maximal $L^p$-regularity]\label{thm:main}
Let $A= -{\rm div}\, a(\cdot) \nabla$  be 
a divergence form elliptic operator on $\R^n$ 
with bounded measurable real-valued coefficients, and denote by $S = (S(t))_{t\ge 0}$ 
the   analytic  $C_0$-contraction semigroup generated by $-A$. Then
for all  $1\le p<\infty$ and $\beta>0$,
and all adapted simple processes 
$g:\R_+\times\R^n\times\Omega\to H$ one has
$$\E \n \nabla S \diamond g \n_{\tbp (\R^n) }^p \le C_{p,\b}^p \E \n g\n_{\tbpH   }^p,$$
with constant $C_{p,\b}$ independent of $g$ and $H$. 
\end{theorem}

\begin{remark}
Compared to the 
results given in \cite{Kry94, Kry06, NVW11}, Theorem \ref{thm:main} gives conical 
stochastic maximal $L^p$-regularity 
for $1\le p<\infty$, while stochastic maximal $L^{p}$-regularity 
 can only hold for $2\le p<\infty$ even for $A = -\Delta$ (see \cite{Kry} and the discussion
in the Introduction). 
\end{remark}

The proof of Theorem \ref{thm:main} combines two ingredients: a $\tb$ estimate, and
an extrapolation result based on off-diagonal bounds for $L$ which gives the $\tbp$ estimate. 
These steps are carried out in Sections \ref{sec:T22} and \ref{sec:extrapolation}, respectively. 

\section{Conical stochastic maximal $L^{2}$-regularity}\label{sec:T22}

A classical stochastic maximal $L^2$-regularity result due to Da Prato 
(see \cite[Theorem 6.14]{DPZ}) asserts that 
if $-A$ generates an analytic $C_0$-contraction semigroup $(S(t))_{t\geq 0}$
on a Hilbert space $E$ and $g$ is an 
  $\F$-adapted simple process with values in the vector space $H\otimes E$
of finite rank operators from $H$ to $E$, then
there exists a constant $C\ge 0$, independent of $g$ and $H$, such that
\begin{align}\label{eq:zero}
\mathbb{E}\n A^{\frac{1}{2}}S\diamond g\n_{L^{2}(\R_+;E)}^2 \leq C^2 
\mathbb{E}\|g\|_{L^{2}(\R_+;\calL_2(H,E))}^2.
\end{align}
Here, $\mathscr{L}_2(H,E)$ denotes the space of Hilbert-Schmidt operators
from $H$ to $E$. 

This estimate has the following weighted analogue.

\begin{proposition}\label{prop:beta}
Suppose $-A$ generates an analytic $C_0$-contraction semigroup $S = (S(t))_{t\geq 0}$
on a Hilbert space $E$. Then for all $\beta \geq 0$ there exists a constant $C_\beta\ge 0$ such that 
for all $\F$-adapted simple processes 
$g: \R_+\times\O\to \calL_2(H,E)$,
$$
\mathbb{E}\|A^{\frac{1}{2}}S\diamond g\|_{L^{2}(\R_+,t^{-\beta}dt;E)}^2 \leq C_\beta^2 
\|g\|_{L^{2}(\R_+,t^{-\beta}dt;\mathscr{L}_{2}(H,E))}^2.
$$
\end{proposition}
\begin{proof} 
For $\beta = 0$, this is Da Prato's result. We thus assume that $\beta>0$.
The proof follows the lines of Theorem \ref{thm:conicalmaxreg} in \cite{AA}.
On the subinterval $(0,\frac{t}{2})$ we estimate, using the It\^o isometry,
\begin{align*}
& \mathbb{E}\Big\| t\mapsto \int_0^{\frac{t}{2}} A^{\frac{1}{2}}S(t-s)g(s)\,d\W(s)
\Big \|_{L^{2}(\R_+,t^{-\beta}dt;E)} ^2
\\ & \qquad = \E \int_0^{\frac{t}{2}}\n A^{\frac{1}{2}}S(t-s)g(s)\|_{\calL_{2}(H,E)} ^{2} 
\,ds\,\frac{dt}{t^{\beta}}
 \\ & \qquad \lesssim \E \int  _{0} ^{\infty} \int  _{0} ^{\frac{t}{2}} (t-s)^{-1}\|g(s)\|_{\calL_{2}(H,E)}  ^{2} \,ds\,\frac{dt}{t^{\beta}}
\\ & \qquad \lesssim \E \int  _{0} ^{\infty}\|g(s)\|_{\calL_{2}(H,E)}  ^{2} \,\frac{ds}{s^\beta}.
\end{align*}
On the subinterval $(\frac{t}{2},t)$ we have, using \eqref{eq:zero},
\begin{align*}
\Big(\mathbb{E}\Big\|t\mapsto \int  _{\frac{t}{2}} ^{t} &A^{\frac{1}{2}}S(t-s)g(s)\,d\W(s)\Big\|_{L^{2}(\R_+,t^{-\beta}dt;E)}^2\Big)^\frac12  \\
&\lesssim \Big(\mathbb{E}\Big\|t\mapsto \int  _{\frac{t}{2}} ^{t} s^{-\frac{\beta}{2}} A^{\frac{1}{2}}S(t-s)g(s)\,d\W(s)\Big\|_{L^{2}(\R_+;E)}^2\Big)^\frac12
\\ & \qquad + \Big(\mathbb{E}\Big\|t\mapsto \int  _{\frac{t}{2}} ^{t} (s^{-\frac{\beta}{2}}-t^{-\frac{\beta}{2}}) A^{\frac{1}{2}}S(t-s)g(s)\,d\W(s)\Big\|_{L^{2}(\R_+;E)}^2\Big)^\frac12
\\ &\lesssim \Big(\mathbb{E}\Big\|t\mapsto \int  _{0} ^{t} s^{-\frac{\beta}{2}} A^{\frac{1}{2}}S(t-s)g(s)\,d\W(s)\|_{L^{2}(\R_+;E)}^2\Big)^\frac12
\\ & \qquad + \Big(\mathbb{E}\Big\|t\mapsto \int  _{\frac{t}{2}} ^{t} (s^{-\frac{\beta}{2}}-t^{-\frac{\beta}{2}}) A^{\frac{1}{2}}S(t-s)g(s)\,d\W(s)\Big\|_{L^{2}(\R_+;E)}^2\Big)^\frac12 \\
&\lesssim \E \|g\|_{L^{2}(\R_+,\frac{ds}{s^{\beta}};\mathscr{L}_{2}(H;E))}
\\ & \qquad +\Big(\mathbb{E}\Big\|t\mapsto \int  _{\frac{t}{2}} ^{t} (s^{-\frac{\beta}{2}}-t^{-\frac{\beta}{2}}) A^{\frac{1}{2}}S(t-s)g(s)\,d\W(s)\Big\|_{L^{2}(\R_+;E)}^2\Big)^\frac12.
\end{align*}
Using once more the It\^o isometry, 
the last part is estimated as follows:
\begin{align*}
& \mathbb{E}\Big\|t\mapsto \int  _{\frac{t}{2}} ^{t} (s^{-\frac{\beta}{2}}-t^{-\frac{\beta}{2}}) 
A^{\frac{1}{2}}S(t-s)g(s)\,d\W(s)\Big\|_{L^{2}(\R_+;E)}^2 
\\ & \qquad  \lesssim \E \int  _{0} ^{\infty} \int  _{\frac{t}{2}} ^{t} 
\frac{|s^{-\frac{\beta}{2}}-t^{-\frac{\beta}{2}}|^2}{|s-t|} \|g(s)\|_{\mathscr{L}_{2}(H,E)} ^{2} \, ds\,dt\\
& \qquad \lesssim \E \int  _{0} ^{\infty} \|g(s)\|_{\mathscr{L}_{2}(H,E)} ^{2} 
\Big(\int  _{s} ^{2s} \frac{|\big(\frac{s}{t}\big)^{\frac{\beta}{2}}-1|^2}{|\frac{t}{s}-1|}\, dt\Big)\,\frac{ds}{s^{\beta+1}}
\\ & \qquad \lesssim \E \|g\|_{L^{2}(\R_+,t^{-\beta}dt;\mathscr{L}_{2}(H,E))}^2.
\end{align*}
The proof is concluded by collecting the estimates. \end{proof}
Following the principles described in Subsection \ref{subsec:conical}, we shall specialise, in the next section, to the case $E= L^{2}(\R^{n})$, and identify
$$L^{2}(\R_+,t^{-\beta}dt;\mathscr{L}_{2}(H,L^{2}(\R^{n}))
= L^{2}(\R_+,t^{-\beta}dt; L^{2}(\R^{n};H)) = \tbH
$$
(cf. \eqref{eq:T22}).

\section{Extrapolating conical stochastic maximal $L^2$-regularity}\label{sec:extrapolation}

In this section, we prove two abstract extrapolation results based on off-diagonal estimates.
Proposition \ref{prop:abstract} is an extrapolation result for 
$p\in [1,\infty)\cap (\frac{2n}{n+2\beta},\infty)$, assuming $L^2$-$L^2$ 
off-diagonal bounds, and Theorem \ref{thm:abstract-q} gives the result for $p\in [1,\infty)$, assuming 
$L^1$-$L^2$ off-diagonal bounds 
(the well-definedness of the stochastic integrals on the left-hand side 
of \eqref{eq:L2-assumption} and \eqref{eq:L2-assumption2} being part of the assumptions).  

\begin{proposition}[Extrapolation via $L^2$-$L^2$ off-diagonal bounds]\label{prop:abstract}
Let $(T_t)_{t>0}$ be a family 
of bounded linear operators on $L^2$, let $\beta>0$, and suppose
there exists a constant $C_\b\ge 0$, independent of $g$ and $H$, such that 
\begin{align}\label{eq:L2-assumption}\mathbb{E}\Big\| 
\int  _{0}^t T_{t-s}g(s,\cdot)\,d\W(s)\Big\|_{\tmb}^{2}
\le C_\b^2 \E\n g\n_{\tmbH}^2
\end{align}
for all $\F$-adapted simple $g:\R_+\times\R^n\times\O\to H$.
If $(t^{\frac12} T_t)_{t>0}$ satisfies $L^2$-$L^2$-off-diagonal 
bounds, then, for 
$p\in [1,\infty) \cap (\frac{2n}{n+2\beta},\infty)$, 
there exists a constant $C_{p,\b}\ge 0$, independent of $g$ and $H$, 
such that
$$\mathbb{E}\Big\| \int  _{0}^t T_{t-s}g(s,\cdot)\,d\W(s)\Big\|_{\tmbp }^{p}
\le C_{p,\b}^p \E \n g\n_{\tmbpH}^p.$$
\end{proposition}

\begin{proof}
We introduce the sets
$$ 
C_j(x,t) = \left\{
\begin{array}{ll}
  B(x,t) &  \ j=0 \\
   B(x,2^jt)\setminus B(x,2^{j-1}t) &  \ j=1,2,\dots
\end{array}
\right.
$$
Fix an $\F$-adapted simple process $g:\R_+\times\R^n\times\O\to H$. 
Using the It\^o isomorphism for stochastic integrals 
\cite{nvw} in combination with a 
square function estimate \cite[Corollary 2.10]{nw}, we obtain 
\begin{equation}\label{eq:comp-Tp2m}
\begin{aligned}
\ & \mathbb{E}\Big\| \int  _{0} ^t T_{t-s}g(s,\cdot)\,d\W(s)\Big\|_{\tmbp }^{p}
\\ & \qquad \eqsim  \E \n \one_{\{t \geq s\}}(t,s)T_{t-s}g(s,\cdot)\n_{T_{2,\beta}^{p,2}(L^2(\R_+;H))}^p 
\\ & \qquad = 
\E  \big\n \,\n \one_{\{t \geq s\}}T_{t-s}g(s,\cdot)\n_{L^2(\R_+;H)}\,\big\n_{\tmbp } ^p
\\ & \qquad
=  \E \int  _{\R^{n}} 
\Big(\int  _{0} ^{\infty} \fint_{B(x,t^\frac12)}
\int_0^t
\n T_{t-s}[g(s,\cdot)](y)\n_{H}^{2}\,ds \,dy\,\frac{dt}{t^{\beta}} \Big)^{\frac{p}{2}} \,dx 
\\ & \qquad \le \E \sum_{j=0}^\infty \sum_{k=1}^\infty I_{j,k} + \E \sum_{j=0}^\infty J_j,
\end{aligned}
\end{equation}
where
$$ I_{j,k} = \int_{\R^n} \!\Big(
 \int_0^\infty\!\! \fint_{B(x,t^\frac12)} \int^{2^{-k}t}_{2^{-k-1}t}
\n T_{t-s} [\one_{C_j(x,4t^\frac12)} g(s,\cdot)](y)\n_H^2\,\,ds\,dy\frac{dt}{t^{\beta}}
\Big)^\frac{p}{2}dx
$$
and 
$$ J_j = \int_{\R^n} \!\Big( 
 \int_0^\infty\!\! \fint_{B(x,t^\frac12)} \int^t_{\frac{t}{2}}
\n T_{t-s} [\one_{C_j(x,4s^\frac12)} g(s,\cdot)](y)\n_H^2\,\,ds\,dy\frac{dt}{t^{\beta}}
\Big)^\frac{p}{2}dx.
$$ 
Following closely the proof given in \cite{AMP} 
we shall estimate each of these contributions separately.

We begin with an estimate for $I_{j,k}$ for $j\ge 0$ and $k\ge 1$. Using the off-diagonal bounds, we find 
\begin{align*}
\ &  \int_0^\infty \!\!\!\fint_{B(x,t^\frac12)} \int^{2^{-k}t}_{2^{-k-1}t}
\n T_{t-s} [\one_{C_j(x,4t^\frac12)} g(s,\cdot)](y)\n_H^2\,\,ds\,dy\frac{dt}{t^{\beta}}
\\ & =  \int_0^\infty\!\!\! \int^{2^{-k}t}_{2^{-k-1}t}
\frac1{t-s}\big\n \one_{B(x,t^\frac12)} (t-s)^\frac12 T_{t-s} 
[\one_{C_j(x,4t^\frac12)} g(s,\cdot)]\big\n_{L^2(H)}^2\,\,ds\frac{dt}{t^{\frac{n}{2}+\beta}}
\\ & \lesssim  \int_0^\infty\!\!\! \int^{2^{-k}t}_{2^{-k-1}t}
\frac1{t} \exp(-\frac{c 4^{j}t}{t-s})\big\n 
\one_{C_j(x,4t^\frac12)} g(s,\cdot)\big\n_{L^2(H)}^2\,\,ds\frac{dt}{t^{\frac{n}{2}+\beta}}
\\ & \lesssim \exp(-c4^{j}) \int_0^\infty\Big( \int^{2^{k+1}s}_{2^k s} \frac{dt}{t^{\frac{n}{2}+1+\beta}}\Big)
\big\n \one_{B(x,2^{j+\frac{k}{2}+3}s^\frac12)} g(s,\cdot)\big\n_{L^2(H)}^2\,ds
\\ & \lesssim \exp(-c4^{j})2^{-k(\frac{n}{2}+\beta)} \int_0^\infty  
\big\n \one_{B(x,2^{j+\frac{k}{2}+3}s^\frac12)} g(s,\cdot)\big\n_{L^2(H)}^2\,\frac{ds}{s^{\frac{n}{2}+\beta}}.
\end{align*}
By \eqref{eq:hnp} it follows that 
\begin{align*}
\E I_{j,k}
& \lesssim \exp(-\tfrac{cp}{2}4^{j})2^{-\frac12k(\frac{n}{2}+\beta)p} 
\E \n g\n_{T_{2,\beta,2^{j+{k}/{2}+3}}^{p,2}(H)}^p
\\ & \lesssim \exp(-\tfrac{cp}{2}4^{j})2^{-\frac12k(\frac{n}{2}+\beta)p}2^{(j+\frac{k}{2}+3)\frac{np}{p\wedge 2}} 
\E \n g\n_{\tmbpH}^p. 
\end{align*}
The sum  $\E \sum_{j,k} I_{j,k}$ thus converges since we assumed that $p> \frac{2n}{n+2\beta}$.

Next we estimate $J_0$. We have 
\begin{align*}
\ &  \E \int_0^\infty \!\!\!\fint_{B(x,t^{\frac{1}{2}})} \int_{\frac{t}{2}}^{t}
\n T_{t-s} [\one_{B(x,4s^\frac12)} g(s,\cdot)](y)\n_H^2\,\,ds\,dy\,\frac{dt}{t^{\beta}}
\\ & \le \E \int_0^\infty \!\!\!\fint_{B(x,t^{\frac{1}{2}})} \int_{0}^{t}
\n T_{t-s} [\one_{B(x,4s^\frac12)} g(s,\cdot)](y)\n_H^2\,\,ds\,dy\,\frac{dt}{t^{\beta}}
\\ & \le \E \int_0^\infty \!\!\!\int_{\R^n} \!\int_{0}^{t}
\n T_{t-s} [\one_{B(x,4s^\frac12)} g(s,\cdot)](y)\n_H^2\,\,ds\,dy\,\frac{dt}{t^{\frac{n}{2}+\beta}} 
\\ &  
= \E \int_0^\infty \!\!\!\int_{\R^n} \E \Big| \int_{0}^{t}
 T_{t-s} [\one_{B(x,4s^\frac12)} g(s,\cdot)](y)\,d\W(s)\Big|^2\,dy\,\frac{dt}{t^{\frac{n}{2}+\beta}} 
\\ &
 =  \E \Big\n (t,y)\mapsto \int_{0}^t
 T_{t-s} [\one_{B(x,4s^\frac12)} g(s,\cdot)](y)\,d\W(s)\Big\n_{L^2(t^{-\frac{n}{2}-\beta}dt\times dy;H)}^2
\\ & \lesssim  \E \n t\mapsto \one_{B(x,4t^\frac12)} g(t,\cdot)\n_{L^2(t^{-\frac{n}{2}-\beta}dt\times dy;H)}^2,
\end{align*}
where the last
inequality follows from 
the $\tmb$-boundedness assumption on the stochastic convolution operator.
It follows that
\begin{align*}
\E J_0
& \lesssim 
\E \int_{\R^n}\n t\mapsto \one_{B(x,4t^\frac12)} g(t,\cdot)
\n_{L^2(t^{-\frac{n}{2}-\beta}dt\times dy;H)}^p\,dx
\\ & = \E \int_{\R^n}\!\Big(\int_0^\infty \int_{\R^n} \one_{B(x,4t^\frac12)}(y) 
\n g(t,y)\n_H^2\,dy 
\,\frac{dt}{t^{\frac{n}{2}+\beta}}\Big)^\frac{p}{2}\,dx
\\ & \lesssim \E \n g\n_{\tmbpH}^p,
\end{align*}
the last of these estimates being a consequence of \eqref{eq:hnp}.

Finally we estimate $J_j$ for $j\ge 1$. We have
\begin{align*}
\ &  \int_0^\infty \!\!\! \fint_{B(x,t^\frac12)} \int^{t}_{\frac{t}{2}}
\n T_{t-s} [\one_{C_j(x,4s^\frac12)} g(s,\cdot)](y)\n_H^2\,\,ds\,dy\frac{dt}{t^{\beta}}
\\ & \lesssim 
\int_0^\infty\!\!\!  \int^{t}_{\frac{t}{2}} \frac1{t-s} \exp(-\frac{c4^{j}s}{t-s})
\n \one_{B(x,2^{j+2}s^\frac12)} g(s,\cdot)\n_{L^2(H)}^2\,\,ds\,\frac{dt}{t^{\frac{n}{2}+\beta}}
\\ & \le
\int_0^\infty \!\!\int^{t}_{\frac{t}{2}}\frac1{t-s} \exp(-c\frac{4^{j}s}{t-s})
\n \one_{B(x,2^{j+2}s^\frac12)} g(s,\cdot)\n_{L^2(H)}^2\,\frac{ds}{s^{\frac{n}{2}+\beta}}\,dt
\\ & = 
\int_0^\infty \!\!\Big(\int^{2s}_{s}\frac1{t-s} \exp(-c\frac{4^{j}s}{t-s})\,dt\Big)
\n \one_{B(x,2^{j+2}s^\frac12)} g(s,\cdot)\n_{L^2(H)}^2\,\frac{ds}{s^{\frac{n}{2}+\beta}}
\\ & \le
\exp(-\frac{c}{2} 4^{j}) \int_0^\infty \!\!\Big( \int_{s}^{2s}\frac1{t-s} 
\exp(-\frac{c}{2}\frac{4^{j}s}{t-s})\,dt\Big)
\n \one_{B(x,2^{j+2}s^\frac12)} g(s,\cdot)\n_{L^2(H)}^2\,\frac{ds}{s^{\frac{n}{2}+\beta}}
\\ & =
\exp(-\frac{c}{2}4^{j}) \int_0^\infty \!\!\Big( \int_1^\infty \exp(-\frac{c}{2}4^{j}u)\,\frac{du}{u}\Big)
\n \one_{B(x,2^{j+2}s^\frac12)} g(s,\cdot)\n_{L^2(H)}^2\,\frac{ds}{s^{\frac{n}{2}+\beta}}
\\ & \lesssim
\exp(-\frac{c}{2}4^{j})\int_0^\infty 
\n \one_{B(x,2^{j+2}s^\frac12)} g(s,\cdot)\n_{L^2(H)}^2\,\frac{ds}{s^{\frac{n}{2}+\beta}}.
\end{align*}

With \eqref{eq:hnp} it follows that
\begin{align*}
\E J_{j}
& \lesssim \exp(-c4^{j-1}p)\E \n g\n_{T_{2,\beta,2^{j+2}}^{p,2}(H)}^p
 \lesssim \exp(-c4^{j-1}p)2^{(j+2)\frac{np}{p\wedge 2}}\E \n g\n_{T_{2,\beta}^{p,2}(H)}^p, 
\end{align*}
and the sum  $\E \sum_j J_j$ thus converges.
\end{proof}

\begin{theorem}[Extrapolation via $L^1$-$L^2$ off-diagonal bounds]\label{thm:abstract-q}
Let $(T_t)_{t>0}$ be a family of bounded linear operators on $L^2$, let $\beta>0$, and suppose
there exists a constant $C_\b\ge 0$,independent of $g$ and $H$, such that 
\begin{align}\label{eq:L2-assumption2} \mathbb{E}\Big\| \int_{0}^t T_{t-s}g(s,\cdot)\,d\W(s)\Big\|_{\tmb}^{2}
\le C_\b^2 \mathbb{E} \n g\n_{\tmbH}^2
\end{align}
for all $\F$-adapted simple process $g:\R_+\times\R^n\times\O\to H$.
If $(t^{\frac12} T_t)_{t>0}$ is a
family of bounded linear operators on $L^2$ 
which satisfies $L^1$-$L^2$ 
off-diagonal bounds, then, for all $p\in [1,\infty)$,
there exists a constant $C_{p,\b}\ge 0$, independent of $g$ and $H$, such that
$$\mathbb{E}\Big\| \int  _{0}^t T_{t-s}g(s,\cdot)\,d\W(s)\Big\|_{\tmbp }^{p}
\le C_{p,\b}^p \E \n g\n_{\tmbpH}^p.$$
\end{theorem}
Recall \,that\,  $L^1$-$L^2$\, off-diagonal\, bounds\, are\, stronger\, than\, $L^2$-$L^2$\, off-diagonal\, bounds by 
Lemma \ref{lem:comparisonoff}, so the previous proposition applies, and gives the result for $p \in [2,\infty)$.

The proof of Theorem \ref{thm:abstract-q} will be based on two  lemmas.
The first gives a simple sufficient condition for membership of $\tmbpH$.

\begin{lemma}\label{lem:Tp2}
If $a\in L^2(\R_+\times\R^n, t^{-\beta}dt\times dy;H)$ is supported in a set of the form
$(0,r^{2})\times B(x_0,r)$ with $r>0$ and $x_0\in \R^n$, then,  for all $1\leq p \leq 2$,
we have $a\in \tmbpH$ and 
$$ \n a\n_{\tmbpH}\lesssim r^{n(\frac1p-\frac12)}\n a\n_{L^2(\R_+\times\R^n,
t^{-\beta}dt\times dy;H)}$$ 
with implied constant depending on $n$ 
and $p$, but not on $\beta$, $r$, and $x_0$.
\end{lemma}

\begin{proof}
Noting that, for $t \in (0,r^{2})$, $B(x_0, r)\cap B(x,t^{\frac12})\not=\emptyset$ only if
$|x-x_0|<t^{\frac12} + r \le 2r$,  from H\"older's inequality we obtain 
\begin{align*}
\n a\n_{\tmbpH}^p & = \int_{\R^n} \Big(\int_0^\infty\fint_{B(x,t^{\frac12})}\n a(t,y)\n_H ^2\,dy 
\frac{dt}{t^{\beta}} \Big)^\frac{p}{2} \,dx
\\ & = \int_{B(x_0,2r)} \Big(\int_0^{r^2}\fint_{B(x,t^{\frac12})}\n a(t,y)\n_H ^2\,dy 
\frac{dt}{t^{\beta}} \Big)^\frac{p}{2} \,dx
\\ & \le \Big(\int_{B(x_0,2r)} \,dx\Big)^{1-\frac{p}{2}}
\Big(\int_{B(x_0,2r)} \int_0^{r^2}\fint_{B(x,t^{\frac12})}\n a(t,y)\n_H ^2\,dy 
\frac{dt}{t^{\beta}}\,dx \Big)^\frac{p}{2} 
\\ & \lesssim 
r^{n(1-\frac{p}{2})}
\Big(\int_{B(x_0,2r)} \int_0^{r^2}\int_{\R^n}\frac{\one_{B(x,t^{\frac12})}(y)}{|B(x,t^{\frac12})|}\n a(t,y)\n_H ^2\,dy 
\,\frac{dt}{t^{\beta}}\,dx \Big)^\frac{p}{2} 
\\ & =
r^{n(1-\frac{p}{2})}
\Big(\int_0^{r^2}\int_{\R^n}\int_{B(x_0,2r)} \frac{\one_{B(y,t^{\frac12})}(x)}{|B(y,t^{\frac12})|}\n a(t,y)\n_H ^2\,dx\,dy 
\,\frac{dt}{t^{\beta}} \Big)^\frac{p}{2} \\ & =
r^{n(1-\frac{p}{2})}
\Big( \int_0^{r^2}\int_{\R^n}\frac{|B(x_0,2r)\cap B(y,t^{\frac12})|}{|B(y,t^{\frac12})|}\n a(t,y)\n_H ^2\,dy 
\,\frac{dt}{t^{\beta}} \Big)^\frac{p}{2} 
\\ & \le r^{n(1-\frac{p}{2})}
\Big(\int_0^{\infty}\int_{\R^n} \|a(t,y)\|_{H}^2 \,dy
\,\frac{dt}{t^{\beta}} \Big)^\frac{p}{2}.
\end{align*}
\end{proof}

For the second lemma we need to introduce some terminology.
An {\em atom} with values in $H$ is a function $a:\R_+\times\R^n\to H$ 
supported in a set of the form $(0,r^{2})\times B(x_0,r)$ 
for some $r>0$ and  $x_0 \in \R^{n}$ and 
satisfying the estimate 
$$
\n a\n_{L^2(\R_+\times\R^n,t^{-\beta}dt\times dy;H)} \leq r^{-\frac{n}{2}}.
$$
By the previous lemma, any atom belongs to $\TmbH$
with norm $\n a\n_{\TmbH}\lesssim 1.$
 
The next lemma is a consequence of the well-known fact that $\TmbH$ admits an atomic decomposition, and interpolation.

\begin{lemma} \label{lem:bdd-on-atoms} 
Let $\beta \in \R$ and let $\mathscr{H}$ be a Hilbert space.
A bounded linear operator from $\tmbH  $ to $\tmb(\mathscr{H})$, which 
is uniformly bounded on atoms, extends to a bounded operator from $\TmbH$ to $\Tmb(\mathscr{H})$.
\end{lemma}
A subtle point here is that an operator that is uniformly bounded on atoms is not necessarily defined on $\TmbH$.
However, if the operator is also bounded on $\tmbH $, 
then a simple modification of \cite[Theorem 4.9, Step 3]{amr} takes care of this issue.

\begin{proof}[Proof of Theorem \ref{thm:abstract-q}]
Given a simple function
$f:\R_+\to L^2\otimes H$, let
$$Mf(t,x) := 
\one_{\{\frac{t}{2} \geq s\}}[T_{t-s}f(s,\cdot)](x).
$$

As in  
\eqref{eq:comp-Tp2m},
given an adapted simple process $g:\R_+\times\Omega\to L^2\otimes H$, for all $1\le p<\infty$, 
we have 
$$
\E \Big\n \int_0^{\frac{t}{2}}  T_{t-s} g(s,\cdot)\,d\W(s)\Big\n^p_{\tmbp } \lesssim
\E \|Mg\|_{\tmbp(L^2(\R_+;H)) }^p.
$$
Hence the theorem is proved once we show that the 
linear mapping $M$ is bounded 
from $\tmbpH$  to $\tmbp(L^2(\R_+;H)) $ for $p\in [1,2]$. 
Indeed, the stochastic integral over the interval 
$(\frac{t}{2},t)$ has already been estimated in the proof of 
Proposition \ref{prop:abstract}.

By interpolation, it suffices to consider the exponents $p=1$ and $p=2$.

\medskip
{\em Step 1} -- 
We start with the case $p=2$. 
Proceeding as in \eqref{eq:comp-Tp2m}, using the isometry $\tmbH = L^2(\R_+\times\R^n,t^{-\beta}dt\times dy;H)$ 
(first with $H$ replaced by $\R$ and at the end of the computation with $H$), Fubini's theorem, the uniform boundedness of 
the operators $t^{\frac12} T_t$, 
we obtain  
\begin{align*}
\|Mf\|_{\tmbp(L^2(\R_+;H))} ^{2} 
&  =  \Big\n \,\big\n \one_{\{\frac{t}{2} \geq s\}}
T_{t-s}f(s,\cdot)\big\n_{L^2(\R_+;H)}\,\Big\n_{\tmbp} ^2
\\ &  = 
\int_0^\infty\int_0^{\frac{t}{2}} 
\n T_{t-s}f(s,\cdot)\n_{L^2(H)}^2\,ds \frac{dt}{t^{\beta}}
\\ & \lesssim    \int  _{0} ^{\infty} \int _{0} ^{\frac{t}{2}} 
\frac{s}{t} \|f(s,\cdot)\|_{L^2(H)}^{2} \,\frac{ds}{s}\frac{dt}{t^{\beta}}
\\ & = \int _{0} ^{\frac{1}{2}} \int  _{0} ^{\infty} 
  \|f(tu,\cdot)\|_{L^2(H)}^{2}\,\frac{dt}{t^{\beta}} \, du
\\ & = \int _{0} ^{\frac{1}{2}} \int  _{0} ^{\infty} 
u^{\beta-1}  \|f(t,\cdot)\|_{L^2(H)}^{2}\,\frac{dt}{t^{\beta}} \, du
\\ & \lesssim  \int _{0} ^{\infty} 
 \|f(t,\cdot)\|_{L^2(H)}^{2} \,\frac{dt}{t^{\beta}}
\\ & = \|f\|_{\tmbpH}^2.
\end{align*}

\medskip
{\em Step 2} --
Next we consider the case $p=1$. 
We will prove that there exists a constant $C_{\b}\ge 0$ such that 
for every atom $a$ we have
\begin{align}\label{eq:atom}
\n M a\n
_{\Tmb(L^2(\R_+;H))} \leq C_{\b}.
\end{align}
An appeal to Lemma \ref{lem:bdd-on-atoms} will then finish the proof. 

Fix an atom $a$  
supported in $(0,r^2)\times B(x_0,r)$, and 
define the following sets:
\begin{equation*}
\begin{split}
C_{0}:=&\{(t,x)\in (0,\infty)\times \mathbb{R}^{n} \;;\; |x-x_0| < 2r \; 
\text{and} \; t<(2r)^2\},\\
C_{j}:=&\{(t,x)\in (0,\infty)\times \mathbb{R}^{n} \;;\; 2^{j}r\leq |x-x_0| < 2^{j+1}r \; 
\text{and} \; t<(2^j r)^2\}, \ j\geq1,\\
C'_{j}:=&\{(t,x)\in (0,\infty)\times \mathbb{R}^{n} \;;\; |x-x_0| < 2^{j+1}r \; 
\text{and} \; (2^j r)^2\leq t<(2^{(j+1)}r)^2\}, \  j\geq1.
\end{split}
\end{equation*}
We write 
\begin{align*}
 \ &
\big\n \one_{\{\frac{t}{2} \geq s\}}
T_{t-s}f(s,\cdot)\big\n_{L^2(\R_+;H)}
\\ & \ \ = \Big(\one_{C_0}(t,x)+\sum  _{j \geq 1} (\one_{C_j}(t,x)+\one_{C'_j}(t,x))\Big)
\Big(\int  _{0} ^{\frac{t}{2}}
\n [T_{t-s}a(s,\cdot)](x)\n_H^{2}\,ds\Big)^\frac12
\end{align*} 
and, using Lemma \ref{lem:Tp2}, show that each term is in $\Tmb$ with suitable bounds.

\noindent {\it 1. Estimate on $C_{0}$:} Estimating as before, using 
the uniform boundedness of $({t}^\frac12 T_{t})_{t>0}$ we have
\begin{align*}
\ & \int  _{B(x_0,2r)} \int  _{0} ^{4r^2} \int  _{0} ^{\frac{t}{2}}  
\n [T_{t-s}a(s,\cdot)](x)\n_H^{2}\, ds\,\frac{dt}{t^{\beta}}\,dx
\\ &\qquad \lesssim \int  _{0} ^{4r^2} \int  _{0} ^{\frac{t}{2}} \frac{s}{t}
\|a(s,\cdot)\|_{L^2(H)}^{2} \,\frac{ds}{s}\,\frac{dt}{t^{\beta}}
\\ &\qquad = \int  _{0} ^{2r^2}\int_{2s}^{4r^2} \frac{s}{t}
 \|a(s,\cdot)\|_{L^2(H)}^{2}\,\frac{dt}{t^{\beta}}\, \frac{ds}{s} 
\\ &\qquad \lesssim \int  _{0} ^{\infty} s^{1 -\beta} 
 \|a(s,\cdot)\|_{L^2(H)}^{2}\, \frac{ds}{s} 
\\ & \qquad \lesssim r^{-n}.
\end{align*}
Therefore, by Lemma \ref{lem:Tp2},
\begin{align*}
\ & \Big\|(t,x)\mapsto \one_{C_0}(t,x)\Big( \int  _{0} ^{\frac{t}{2}} 
\n T_{t-s}[a(s,\cdot)](x)\n_H ^{2}\,ds\Big)^{\frac{1}{2}}\Big\|_{\Tmb}  \lesssim 1.
\end{align*}

\noindent {\it 2. Estimate on $C_{j}$:} 
Let us write $\tilde B_j = {B(x_0,2^{j+1}r)}\setminus {B(x_0,2^{j}r)}$.
Using the $L^1$-$L^2$ off-diagonal estimates, and the fact that $a$ is supported
on $(0,r^2)\times B(x_0,r)$, we have 
\begin{align*}
& \int  _{\tilde B_j} \int  _{0} ^{4^{j}r^{2}} \int  _{0} ^{\frac{t}{2}} 
\n T_{t-s}[a(s,\cdot)](x)\n_H^{2}\,ds\,\frac{dt}{t^{\beta}}\,dx\\
& = \int  _{0} ^{4^{j}r^{2}} \int  _{0} ^{\frac{t}{2}} \int  _{\tilde B_j}   
\n{t}^\frac12T_{t-s}[a(s,\cdot)\one_{B(x_0,r)}(\cdot)](x)\n_H^{2}\,dx\,ds\,\frac{dt}{t^{1+\beta}}\\
& \lesssim \int  _{0} ^{4^{j}r^{2}} \int  _{0} ^{\frac{t}{2}}  t^{-\frac{n}{2}}
\exp(-c{4^{j}r^{2}}/{t})\|a(s,\cdot)\|_{L^1(\R^n;H)}^{2}\, ds\,\frac{dt}{t^{1+\beta}}\\
& =
\int  _{0} ^{4^{j}r^{2}} \Big(\int  _{0} ^{\frac{t}{2}}  s^{\beta}\|a(s,\cdot)\|_{L^1(\R^n;H)} ^{2} 
\,\frac{ds}{s^{\beta}}\Big)  
t^{-\beta-\frac{n}{2}}
\exp(-c{4^{j}r^{2}}/{t})\,\frac{dt}{t}
\\
&\lesssim \Big(r^{n}\int  _{0} ^{r^2} r^{2\beta}\|a(s,\cdot)\|_{L^2(H)} ^{2} \,\frac{ds}{s^{\beta}}\Big) 
 \int  _{0} ^{4^{j}r^{2}}  t^{-\beta-\frac{n}{2}}
\exp(-c{4^{j}r^{2}}/{t})\,\frac{dt}{t}
\\ & \lesssim 
r^{2\beta}
(4^{j}r^{2})^{-\beta-\frac{n}{2}}
\\ & = 
r^{-n} 4^{-j(\beta+\frac{n}{2})}.
\end{align*}
Therefore, by Lemma \ref{lem:Tp2},
\begin{align*}
\  \Big\|(t,x)\mapsto \one_{C_j}(t,x)&\Big( \int  _{0} ^{\frac{t}{2}} 
\n T_{t-s}[a(s,\cdot)](x)\n_H ^{2}\,ds\Big)^{\frac{1}{2}}\Big\|_{\Tmb} 
\\ &  \lesssim   (2^{j+1}r)^{\frac{n}{2}}r^{-\frac{n}{2}}2^{-j(\beta+\frac{n}{2})}
\lesssim 2^{-j\beta}.
\end{align*}

\noindent {\it 3. Estimate on $C'_{j}$:} 
Using the $L^1$-$L^2$ off-diagonal bounds, we have 
\begin{align*}
\ & \int  _{B(x_0,2^{j+1}r)} \int  _{(2^j r)^2} ^{(2^{j+1}r)^2} \int  _{0} ^{\frac{t}{2}}  
\n T_{t-s}[a(s,\cdot)](x)\n_H^{2}\, ds\, \frac{dt}{t^{\beta}}\, dx
\\ &\qquad  =  \int  _{(2^j r)^2} ^{(2^{j+1}r)^2} \int  _{0} ^{\frac{t}{2}} \int  _{B(x_0,2^{j+1}r)}\n {t}^\frac12 T_{t-s}[a(s,\cdot)\one_{B(x_0,r)}(\cdot)](x)\n_H ^{2}\, dx\, ds\,\frac{dt}{t^{1+\beta}}
\\ & \qquad \lesssim \int  _{(2^j r)^2} ^{(2^{j+1}r)^2} \int  _{0} ^{\frac{t}{2}}  
t^{-\frac{n}{2}}
\|a(s,\cdot)\|_{L^1(\R^n;H)}^{2} \,ds\,\frac{dt}{t^{1+\beta}}\\
& \qquad \lesssim \Big(r^{n}\int  _{0} ^{r^2} r^{2\beta}\|a(s,\cdot)\|_{L^2(H)} ^{2}\,\frac{ds}{s^{\beta}}\Big) 
 \int   _{(2^j r)^2} ^{(2^{j+1}r)^2}  t^{-\beta-\frac{n}{2}}
\,\frac{dt}{t} 
\\ & \qquad \lesssim r^{2\beta}
(4^{j}r^{2})^{(-\beta -\frac{n}{2})}
\\ & \qquad = 
r^{-n} 4^{-j(\beta+\frac{n}{2})}.
\end{align*}
Therefore, by Lemma \ref{lem:Tp2}, 
\begin{align*}
\ & \|(t,x)\mapsto \one_{C'_j}(t,x) \Big(\int  _{0} ^{\frac{t}{2}} 
\n T_{t-s}[a(s,\cdot)](x)\n_H^{2}\, ds\Big)^{\frac{1}{2}}\|_{\Tmb} \lesssim 2^{-j\beta}.
\end{align*}

\noindent {\it 4. Collecting the estimates:} 
Summing the above three estimates over $j$ gives \eqref{eq:atom}.
\end{proof}

\begin{proof}[Proof of Theorem \ref{thm:main}]
For $A$ as in Example \ref{ex:divform}, $(t^{\frac12}\nabla e^{-tA})_{t \geq 0}$ has $L^1$-$L^2$ off-diag\-onal bounds. 
By the solution of Kato's square root problem \cite{ahlmt},
\begin{align}\label{eq:ksqrt}
\|A^{\frac{1}{2}}u\|_{L^{2}} \eqsim \|\nabla u\|_{L^{2}(\R^n)}.
\end{align}
Moreover, $A$  is maximal accretive on $L^2$ and therefore
the bounded analytic semigroup generated by $-A$ is contractive on $L^2$.
By Proposition \ref{prop:beta}, the mapping $g \mapsto A^{\frac{1}{2}} S \diamond g$ (and hence the mapping 
$g \mapsto \nabla S \diamond g$) extends to a bounded operator from
$L^2_\F (\O;\tmbH)$ to $L^2_\F(\O;\tmb)$ (respectively, from $L^2_\F (\O;\tmbH)$ to $L^2_\F(\O;\tmb(\R^n))$). 
The result thus follows from Theorem \ref{thm:abstract-q}.
\end{proof}

\begin{remark}
The results in this section are stated in way that is suitable for applications to the divergence form elliptic operators from Example \ref{ex:divform}. 
Introducing an homogeneity parameter $m$ as in \cite{AKMP,AMP}, one can prove analogue results suitable for the study of differential operators of order $m$. We have chosen not to do so here to make the paper more readable.
\end{remark}

\section{An application to SPDE}\label{sec:SPDE}

In this section we apply our results to prove conical stochastic maximal $L^p$-regularity
for a class of nonlinear stochastic evolution equations.
We consider the problem
\begin{equation}\label{eq:nonlinear} \left\{
\begin{aligned}
du(t,x) &= {\rm div}\, a(x) \nabla u(t,x) \,dt + b(\nabla u(t,x))\,dW(t), && t\ge 0, \ x\in\R^n, \\ 
u(0,x) &= u_0(x), && x\in\R^n. 
\end{aligned}
\right.
\end{equation}
Here, $W$ is an $\F$-Brownian motion relative to some given filtration $\F$, 
the function $a: \R^n \to M^n(\R)$ is bounded and measurable, the operator $A=-{\rm div}\,a \nabla$ satisfies the ellipticity conditions of Example \ref{ex:divform},
the function $b: \R^n \to \R$ is globally Lipschitz continuous, with Lipschitz constant $L_b$, 
and satisfies 
\begin{align}\label{eq:ass-b}
|b(x)|\le C_b|x|, \quad x\in\R^n.
\end{align} The initial value  
$u_{0}:\R^n \to \R$ belongs to $\Dom_{p}(A^{\frac{\beta}{2}})$, 
the domain of the $L^p$-realisation of $A^{\frac{\beta}{2}}$, for some $0<\beta<1$. 
At the expense of making the arguments more involved, we could also add an additional 
semilinear term and consider cylindrical Brownian motions,
but in order to bring out the principles more clearly we have chosen to consider a simple
model problem. 

In order to arrive at a notion of solution we proceed as follows.
At least formally, we reformulate \eqref{eq:nonlinear} as an abstract initial value problem 
as follows:

\begin{equation*} \left\{
\begin{aligned}
dU(t) + A U(t) \,dt & = B(\nabla U(t))\,dW(t), && t\ge 0,\\ 
 U(0) &= u_0.
\end{aligned}
\right.
\end{equation*}
Here $A= -{\rm div}\, a(\cdot) \nabla$
and $$ (B(u))(t,x) := b(u(t,x))$$ is the Nemytskii operator associated with $b$. 
Denoting by $L_\F^p(\Omega;\tbp)$ the closed subspace of all $\F$-adapted
processes belonging to $L^p(\Omega;\tbp)$, 
it is immediate from \eqref{eq:ass-b} that $B$ maps $L_\F^p(\Omega;\tbp)$ into itself,
and the Lipschitz continuity of $b$ implies that of $B$, with the same constant.

In order to be consistent with the terminology used in the Introduction,
at least formally, a ``mild solution'' should be an adapted ``process'' $U$ 
that ``satisfies'' the variation of constants equation
\begin{equation}\label{eq:abstr-nonlinear} 
U(t) = S(t)u_0 + \int_0^t S(t-s)B(\nabla U(s))\,dW(s), 
\end{equation}
where $S$ is the bounded analytic $C_0$-semigroup generated by $-A$. 
By {\em conical stochastic maximal $L^p$-regularity} we then understand that 
the ``gradient'' of  $U$ is in $L^p(\Omega;\tbp(\R^n))$. 
In order to make this rigorous,
we formally apply $\nabla$ to both sides of the identity \eqref{eq:abstr-nonlinear}
 and, again formally,  substitute $V = \nabla U$ to arrive at
the equation 
\begin{equation}\label{eq:fp}
 V  
= \nabla S(\cdot)u_0 + \nabla S\diamond B(V).
\end{equation}
\begin{definition}
The problem \eqref{eq:nonlinear} is said to have {\em conical stochastic maximal $L^p$-regularity with weight $\beta$}
if for every initial value $u_0\in \Dom_p(A^\frac{\b}{2})$ 
there exists a unique element $V$ in $L^p(\Omega;\tbp(\R^n))$ such that
\eqref{eq:fp} holds.
\end{definition}

Thus we solve for $V$, rather than for $U$. The above heuristic discussion shows that we may think
of $V$ as the ``gradient of the mild solution of \eqref{eq:abstr-nonlinear}''.

\begin{remark} 
If $V$ solves \eqref{eq:fp}, then, at least formally, we have
$V = \nabla U$ with $U := S(\cdot)u_0 + S \diamond B(V).$  
This definition makes sense provided stochastic convolution on the right-hand side is well defined
in one way or the other. We are not asserting, however, that this process 
is a ``mild solution'' to \eqref{eq:abstr-nonlinear} in any rigorous sense.
\end{remark}

In the next lemma, which is of interest in its own right, we denote by 
$S$ the semigroup generated by $-A$ on $L^p$. 

\begin{lemma}\label{lem:u0} There exists $\beta_{0}\in (0,1]$ with the following property. 
If $p \in (1,\infty)$ and $0<\beta<1$ are such that the pair $(\frac{1}{p},\beta)$ belongs to 
the interior of the planar polytope with vertices 
$(0,0), (0, \beta_{0}), (\frac12, 1), (1,1), (1,0)$, then
for all $u_0\in \Dom_{p}(A^{\frac{\beta}{2}})$ the function
$(t,x)\mapsto \nabla S(t)u_0(x)$ belongs to $\tbp(\R^n)$.  
\end{lemma}

\begin{figure}
\begin{center}
\includegraphics{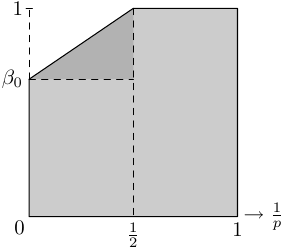}
\medskip
\begin{tabbing}
\hskip1.7cm The polytope of Lemma \ref{lem:u0}. Shaded in gray are the regions \\
\hskip1.7cm of steps 3 and 2; shaded dark is the region of step 4.
\end{tabbing}
\end{center}
\end{figure}

\begin{remark}\label{rem:beta0}
 The constant $\beta_0$ is related to elliptic regularity theory, and can be
arbitrary small. For certain specific classes of operators $A$, however, 
it is known that one can take $\beta_0 = 1$, viz. in the case of constant coefficients 
(in this case, the arguments can be simplified using standard
Littlewood-Paley estimates), and in the case of continuous periodic coefficients with common periods
(see \cite[Section 5.4]{A}, \cite[page 139]{at}, and references therein).
\end{remark}

\begin{remark}
For $\beta=0$, the lemma holds for $1<p<\infty$ using \cite[Theorem 3.1]{ahm} and 
$p_{-}(A)=1$ (the number $p_{-}(A)$ being defined in \cite{ahm}).  The argument given here for 
$1<p<2$ and $0<\beta<1$ 
applies to $\beta=0$ as well, and gives a different proof of this case. 
\end{remark}

\begin{proof}
For this proof, we use complexified spaces.  Let $v_{0} = A^{\frac{\beta}{2}}u_{0}$.
The proof proceeds in four steps.

{\em Step 1} --
In the case $p=2$ and $0<\beta<1$, one has 
\begin{align*}
\|(t,y) \mapsto \nabla S(t)u_{0}(y)\|^2_{\tbp(\C^n)} 
&\eqsim  \int_0^\infty\int_{\R^n} |\nabla e^{-tA} u_{0}|^2 \,dy\, \frac{dt}{t^\beta}
\\
& \eqsim \int_0^\infty\int_{\R^n} |A^{\frac 1 2} e^{-tA} u_{0}|^2 \,dy\, \frac{dt}{t^\beta}
\\
&
=\int_0^\infty\int_{\R^n} |(tA)^{\frac {1-\beta} 2} e^{-tA} v_{0}|^2 \,dy\, \frac{dt}{t}
\\
&
\eqsim \|v_{0}\|_{2}^2
\end{align*}
where the second equivalence uses the equivalence 
$\n A^{\frac 1 2}u\n_{L^2}\eqsim \n \nabla u\n_{L^2(\R^n)}$
(cf. \eqref{eq:ksqrt}), and
the last inequality uses the boundedness of the $H^{\infty}$-functional calculus of $A$ in $L^{2}$ 
(a result going back to \cite{mc}, see also \cite[Theorem 11.9]{KW}).

{\em Step 2} --
In the case $1<p<2$ and $0<\beta<1$, we use the first-order approach of \cite{HNP}.
Consider the operator $$D = \left( \begin{array}{cc} 0 & -\textrm{div}\, a(\cdot) \\ \nabla & 0 \end{array} \right)$$ 
acting on $L^{2}\oplus L^{2}(\C^{n})$. This operator is bisectorial with angle 
$\eta\in (0,\frac{\pi}{2})$,   
its resolvent satisfies $L^{2}$-$L^{2}$ off-diagonal estimates with arbitrary large decay and  
has a bounded $H^{\infty}$ functional calculus. 
The operator $A$ is sectorial of angle $\omega\in (0, \frac \pi 2)$ and has the same properties. 
In particular, $-A$ generates a bounded analytic 
 $C_0$-semigroup on $L^{2}\oplus L^{2}(\C^{n})$ (see \cite{akm}). 

Note that $$t^{\frac{1-\beta}{2}} \nabla S(t)u_{0}
= t^\frac12\nabla e^{-tA}(tA)^{-\frac{\beta}{2}}v_{0}$$ and 
$$
(0, t^\frac12\nabla e^{-tA}(tA)^{-\frac{\beta}{2}}v_{0})= \psi(\sqrt{t}D)(v_{0},0)
$$
for the function
$\psi \in H^{\infty}(S_{\theta})$ defined by $$\psi(z) = z(z^{2})^{-\frac{\beta}{2}}e^{-z^{2}}$$
(where $\theta \in (\omega,\frac{\pi}{2})$, 
$S_{\theta} = \Sigma_{\theta}\cup (-\Sigma_{\theta})\cup\{0\}$, and $\Sigma_{\theta} = 
\{z \in \C \backslash \{0\} \; |\arg(z)| < \theta\}$).

Therefore, by \cite[Theorem 7.10]{HNP}, one has the following equivalence of norms in the tent spaces 
upon changing $\psi$ in the appropriate class (more precisely, as $1<p<2$, 
it is enough that $\psi$ has slight decay at 0 and sufficiently large decay at $\infty$: 
see the condition (d) in Proposition 7.5 there and note that we have cotype $2$), 
$$
 \|(t,y) \mapsto \psi(\sqrt{t}D)(v_{0},0)(y)\|_{T^{p,2}_{2,1}(\C^{n+1})} \eqsim
 \|(t,y) \mapsto \widetilde{\psi}(\sqrt{t}D)(v_{0},0)(y)\|_{T^{p,2}_{2,1}(\C^{n+1})} 
 $$
for $\widetilde{\psi}(z) = ze^{-z^{2}}$.
As
$$
\widetilde{\psi}(\sqrt{t}D)(v_{0},0)=(0, t^{\frac12} \nabla S(t)v_{0})
$$
 we have shown that
\begin{equation*}
\begin{split}
\|(t,y) \mapsto \nabla S(t)u_{0}(y)\|_{\tbp(\C^n)} 
&= \|(t,y) \mapsto t^{\frac{1-\beta}{2}}\nabla S(t)A^{-\frac{\beta}{2}}v_{0}(y)\|_{T^{p,2}_{2,1}(\C^n)}\\
&\eqsim \|(t,y) \mapsto t^{\frac12} \nabla S(t)v_{0}(y)\|_{T^{p,2}_{2,1}(\C^n)}
\\ & = \|(t,y) \mapsto \nabla S(t)v_{0}(y)\|_{T^{p,2}_{2,0}(\C^n)}.
\end{split}
\end{equation*}
Using \cite[Theorem 3.1]{ahm}, this gives
$$
\|(t,y) \mapsto \nabla S(t)u_{0}(y)\|_{\tbp(\C^n) } \lesssim \|v_{0}\|_{L^{p}} 
= \|A^{\frac{\beta}{2}}u_{0}\|_{L^{p}}.
$$

{\em Step 3} --
We turn to the case $2<p<\infty$ and $0<\beta<\beta_{0}$, where $\beta_{0} \in (0,1]$ 
will be determined in a moment. 
We use the fact that  the spaces $\tbp$ interpolate (by either the complex or the real method, 
see e.g. \cite{CMS}) between $p=2$ and $p=\infty$. For $p=\infty$, $\tbp$ is defined as the 
space of all locally square integrable functions such that  the Carleson measure condition
$$
\int_{0}^{r^2}\int_{B} |g(t,y)|^2 \,dy\, \frac{dt}{t^\beta} \le C r^n 
$$
holds whenever $B$ is a ball of radius $r>0$, with $C$ independent of $B$. 

We claim that there exist $\beta_{0}\in (0,1]$ such that for all $0<\beta<\beta_{0}$ 
and $f \in L^\infty$, 
$$
\|(t,y) \mapsto t^{\frac{1-\beta}{2}}\nabla A^{-\frac{\beta}{2}}e^{-tA}f(y)\|_{T^{\infty,2}_{2,1}(\C^n)} 
\lesssim \|f\|_{L^\infty}. 
$$
Assuming the claim, and using the $p=2$ result for all $0<\beta<1$, one concludes, by interpolation, 
that for
$0<\beta<\beta_{0}$ and $p \in [2,\infty)$, 
$$
\|(t,y) \mapsto t^{\frac{1-\beta}{2}}\nabla A^{-\frac{\beta}{2}}e^{-tA}v_{0}(y)\|_{T^{p,2}_{2,1}(\C^n)} 
\lesssim \|v_{0}\|_{L^p},
$$
i.e.,
$$
\|(t,y) \mapsto \nabla S(t)u_{0}(y)\|_{\tbp(\C^n) }  \lesssim \|A^{\frac{\beta}{2}}u_{0}\|_{L^{p}}.
$$

We now prove the claim. 
The argument is scale and translation invariant (up to changing the matrix 
$a(x)$ to $a(rx+x_{0})$ which does not change the ellipticity constants), so we assume that $B$ is the 
unit ball. The same proof works in the general case, and produces the required factor $r^{n}$.
Let $f$ be a bounded measurable function with compact support. 
Let $f_{0}=f1_{2B}$ and $f_{1} = f - f_{0}$. 
By the $p=2$ result, one has
$$
\int_{0}^{1}\int_{B} |t^{\frac{1-\beta}{2}}\nabla A^{-\frac{\beta}{2}}e^{-tA}f_{0}(y)|^2 \,dy \,\frac{dt}{t}
\lesssim \|f_{0}\|_{L^2}^2 \lesssim \|f\|_{L^\infty}^2. 
$$
To control the other term, we use the representation formula
$$
\nabla A^{-\frac{\beta}{2}}e^{-tA}f_{1}= C \int_{0}^\infty 
\nabla s^{\beta/2} (sA) e^{-(s+t)A}f_{1} \frac{ds}{s},
$$
for some constant $C>0$ independent of $f_1$, which holds in $L^2(B;\C^n)$
thanks to the following estimate 
on the kernel $\tilde K_{s}(x,y)$  of $Ae^{-sA}$.
There are constants $c,C>0$ and $\gamma_{0}>0$ in $ (n-2, n]$ (for $n=1$, one has $\gamma_{0}=1$) 
such that 
$$
\forall y \not \in 2B \quad
\bigg(\int_{B}  |\nabla_{x} \tilde K_{s+t}(x,y)|^2\, dx\bigg)^{\frac12} 
\le \frac{C}{(s+t)^{\frac{n}{4} + \frac32+ \frac{\gamma_{0}}{4}}} 
\cdot \exp\Big(\!-\frac{c|y|^2}{s+t}\Big).$$
This estimate is proven in  \cite[Lemma 33, p. 139]{at}  for $e^{-tA}$ but the same proof 
(based on estimates for the kernel of the resolvent) gives the estimate
for $Ae^{-sA}$.
Define $\beta_{0} =\frac{1}{2}( \gamma_{0}-n+2)\in (0,1]$. By 
the kernel estimate and the fact that $f_{1}$ is supported away from $2B$, we have 
\begin{equation*}
\begin{split}
& \| \nabla A^{-\frac{\beta}{2}}e^{-tA}f_{1}\|_{L^2(B;\C^n)} \\
& \qquad 
\eqsim \Big\n\int_{0}^\infty 
\nabla s^{\frac{\beta}{2}} (sA) e^{-(s+t)A}f_{1} \frac{ds}{s}\Big\n_{L^2(B;\C^n)}
          \\
& \qquad \le  \int_{0}^\infty s^{1+\frac{\beta}{2}}  \int_{|y|\ge 2}
\n \nabla\tilde K_{s+t}(\cdot, y)f_{1}(y) \Big\n_{L^2(B;\C^n)}\,dy\,\frac{ds}{s} 
          \\
& \qquad \lesssim  \int_{0}^\infty  \int_{|y|\ge 2}\frac{s^{1+\frac{\beta}{2}}}
{(s+t)^{\frac{n}{4} + \frac32+ \frac{\gamma_{0}}{4}}} 
\cdot \exp\Big(\!\!-\!\frac{c|y|^2}{s+t}\Big)
|f_1(y)|\,dy\, \frac{ds}{s}\\
& \qquad \lesssim \|f_{1}\|_{L^\infty} \int_{|y|\ge 2} \frac{ 1}{|y|^{n+\beta_{0}-\beta}}\,dy \\
& \qquad \lesssim \|f_{1}\|_{L^\infty},
\end{split}
\end{equation*}
for $0<\beta<\beta_{0}$. Therefore
$\int _{0} ^{1}  \|t^{\frac{1-\beta}{2}} \nabla A^{-\frac{\beta}{2}}e^{-tA}f_{1}\|_{L^{2}(B;\C^n)} ^{2} 
\frac{dt}{t} \lesssim  \|f\|_{L^\infty}^2.$

{\em Step 4} -- 
The arguments so far show that the $\tbp(\C^n)$ estimate holds for $0< \Re \beta <1$ and $p=2$, 
with controlled growth in terms of $\Im \beta$. It also shows that the $\tbp(\C^n)$ estimate 
holds for $0< \Re \beta <\beta_{0}$ and $p=\infty$ again with controlled growth in $\Im \beta$. 
Using Stein's complex interpolation, one also gets the conclusion of the lemma for all pairs $(p,\beta)$ 
such that the point $(\frac1p, \Re \beta)$ is inside the planar polytope with vertices 
$(0,0), (0, \beta_{0}), (\frac12, 1), (\frac12, 1)$, $(0,1)$. We leave the details to the reader.
\end{proof}

\begin{theorem}\label{thm:SE} 
Let $p \in (1,\infty)$ and $0<\beta<1$ be such that the pair $(\frac{1}{p},\beta)$ belongs to 
the interior of the planar polytope with vertices 
$(0,0), (0, \beta_{0}), (\frac12, 1), (1,1), (1,0)$, where $\beta_{0}$ is defined in Lemma \ref{lem:u0}.
Suppose that $ K_{p,\beta} L_b <1,$ where $K_{p,\beta}$ 
is the norm of the mapping $g\mapsto \nabla S\diamond g$ from $L^p(\O;\tbp)$ to $L^p(\O;\tbp(\C^n))$ and  
$L_b$ is the Lipschitz constant of $b$. Then 
the problem \eqref{eq:nonlinear} has conical stochastic maximal $L^p$-regularity 
with weight $\beta$,  i.e., \eqref{eq:fp} holds 
for all initial values $u_0\in \Dom_{p}(A^{\frac{\beta}{2}})$.
\end{theorem}

\begin{proof}
Consider the fixed point mapping  
$F$ on $ L_\F^p(\Omega; \tbp(\R^n)  )$ defined by
$$ F(v):= \nabla S(\cdot)u_0 +  \nabla S \diamond B(v).$$
By Theorem \ref{thm:abstract-q}, we have  $\nabla S \diamond B(v) \in
 L_\F^p(\Omega;\tbp(\R^n)  )$ for 
all $v\in \tbp(\R^n)  $.
This, in combination with the previous lemma, 
shows that $F$ maps $L_\F^p(\Omega;\tbp(\R^n)  )$ into itself. 

For $v_1,v_2\in L_\F^p(\Omega;\tbp(\R^n)  )$ we may estimate
\begin{align*}
\|F(v_1)-F(v_2)\|_{L^p(\Omega;\tbp(\R^n)  )} & =  \|\nabla S\diamond (B(v_1) -
B(v_2))\|_{L^p(\Omega;\tbp(\R^n)  )}
\\ & \leq K_{p,\beta} \|B(v_1) - B(v_2)\|_{L^p(\Omega;\tbp(\R^n)  )}
\\ & \leq K_{p,\beta} L_{b} \|v_1 -v_2\|_{L^p(\Omega;\tbp(\R^n)  )}. 
\end{align*}
Since by assumption $K_{p,\beta} L_{b} <1$,  $F$
has a unique fixed point $V$ in $L_\F^p(\Omega;\tbp(\R^n) )$ 
and the theorem is proved.
\end{proof}

\providecommand{\href}[2]{#2}



\begin{thebibliography}{10}

\bibitem{A}
P.~Auscher, \emph{On necessary and sufficient conditions for {$L^p$}-estimates
  of {R}iesz transforms associated to elliptic operators on {$\Bbb R^n$} and
  related estimates}, Mem. Amer. Math. Soc. \textbf{186} (2007), no.~871,
  xviii+75.

\bibitem{a-angle}
\bysame, \emph{Change of angle in tent spaces}, C. R. Math. Acad. Sci. Paris
  \textbf{349} (2011), no.~5-6, 297--301.

\bibitem{AA}
P.~Auscher and A.~Axelsson, \emph{Remarks on maximal regularity}, Herbert Amann
  Festschrift, Progress in Nonlinear Differential Equations and Their
  Applications, vol.~80, Birkh\"auser Verlag, 2011, pp.~45--56.

\bibitem{AA-Inventiones}
\bysame, \emph{Weighted maximal regularity estimates and solvability of
  non-smooth elliptic systems {I}}, Invent. Math. \textbf{184} (2011), no.~1,
  47--115.

\bibitem{ahlmt}
P.~Auscher, S.~Hofmann, M.~Lacey, A.~McIntosh, and Ph. Tchamitchian, \emph{The
  solution of the {K}ato square root problem for second order elliptic
  operators on $\mathbb{R}^n$}, Ann. of Math. \textbf{156} (2002), no.~2,
  633--654.

\bibitem{ahm}
P.~Auscher, S.~Hofmann, and J.-M. Martell, \emph{Vertical versus conical square
  functions}, Trans. Amer. Math. Soc. \textbf{364} (2012), no.~10, 5469--5489.
  \MR{2931335}

\bibitem{AKMP}
P.~Auscher, C.~Kriegler, S.~Monniaux, and P.~Portal, \emph{Singular integral
  operators on tent spaces}, J. Evol. Equ. \textbf{12} (2012), no.~4, 741--765.
  \MR{3000453}

\bibitem{am}
P.~Auscher and J.M. Martell, \emph{Weighted norm inequalities, off-diagonal
  estimates and elliptic operators. {II}. {O}ff-diagonal estimates on spaces of
  homogeneous type}, J. Evol. Equ. \textbf{7} (2007), no.~2, 265--316.

\bibitem{amr}
P.~Auscher, A.~McIntosh, and E.~Russ, \emph{Hardy spaces of differential forms
  on {R}iemannian manifolds}, J. Geom. Anal. \textbf{18} (2008), no.~1,
  192--248.

\bibitem{AMP}
P.~Auscher, S.~Monniaux, and P.~Portal, \emph{The maximal regularity operator
  on tent spaces}, Commun. Pure Appl. Anal. \textbf{11} (2012), no.~6,
  2213--2219. \MR{2912744}

\bibitem{at}
P.~Auscher and Ph. Tchamitchian, \emph{Square root problem for divergence
  operators and related topics}, Ast\'erisque (1998), no.~249, viii+172.

\bibitem{akm}
Andreas Axelsson, Stephen Keith, and Alan McIntosh, \emph{Quadratic estimates
  and functional calculi of perturbed {D}irac operators}, Invent. Math.
  \textbf{163} (2006), no.~3, 455--497.

\bibitem{bk}
S.~Blunck and P.C. Kunstmann, \emph{Calder\'on-{Z}ygmund theory for
  non-integral operators and the {$H\sp \infty$} functional calculus}, Rev.
  Mat. Iberoamericana \textbf{19} (2003), no.~3, 919--942.

\bibitem{CMS}
R.R. Coifman, Y.~Meyer, and E.M. Stein, \emph{Some new function spaces and
  their applications to harmonic analysis}, J. Funct. Anal. \textbf{62} (1985),
  no.~2, 304--335.

\bibitem{DP}
G.~Da~Prato, \emph{Regularity results of a convolution stochastic integral and
  applications to parabolic stochastic equations in a {H}ilbert space}, Confer.
  Sem. Mat. Univ. Bari (1982), no.~182, 17.

\bibitem{DPZ}
G.~Da~Prato and J.~Zabczyk, \emph{Stochastic equations in infinite dimensions},
  Encyclopedia of Mathematics and its Applications, vol.~44, Cambridge
  University Press, Cambridge, 1992.

\bibitem{dk}
B.E.J. Dahlberg and C.E. Kenig, \emph{Hardy spaces and the {N}eumann problem in
  {$L\sp p$} for {L}aplace's equation in {L}ipschitz domains}, Ann. of Math.
  (2) \textbf{125} (1987), no.~3, 437--465.

\bibitem{HTV}
E.~Harboure, J.L. Torrea, and B.E. Viviani, \emph{A vector-valued approach to
  tent spaces}, J. Analyse Math. \textbf{56} (1991), 125--140.

\bibitem{HNP}
T.P. Hyt{\"o}nen, J.M.A.M.~van Neerven, and P.~Portal, \emph{Conical square
  function estimates in {UMD} {B}anach spaces and applications to
  {$H^\infty$}-functional calculi}, J. Anal. Math. \textbf{106} (2008),
  317--351.

\bibitem{kl}
H.~Koch and T.~Lamm, \emph{Geometric flows with rough initial data}, Asian J.
  Math. \textbf{16} (2012), no.~2, 209--235. \MR{2916362}

\bibitem{kt}
H.~Koch and D.~Tataru, \emph{Well-posedness for the {N}avier-{S}tokes
  equations}, Adv. Math. \textbf{157} (2001), no.~1, 22--35.

\bibitem{Kry94}
N.V. Krylov, \emph{A generalization of the {L}ittlewood-{P}aley inequality and
  some other results related to stochastic partial differential equations},
  Ulam Quart. \textbf{2} (1994), no.~4, 16 ff.

\bibitem{Kry}
\bysame, \emph{An analytic approach to {SPDE}s}, Stochastic partial
  differential equations: six perspectives, Math. Surveys Monogr., vol.~64,
  Amer. Math. Soc., 1999, pp.~185--242.

\bibitem{Kry06}
\bysame, \emph{On the foundation of the {$L\sb p$}-theory of stochastic partial
  differential equations}, Stochastic partial differential equations and
  applications---{VII}, Lect. Notes Pure Appl. Math., vol. 245, Chapman \&
  Hall/CRC, Boca Raton, FL, 2006, pp.~179--191.

\bibitem{KW}
P.C. Kunstmann and L.W. Weis, \emph{Maximal {$L\sb p$}-regularity for parabolic
  equations, {F}ourier multiplier theorems and {$H\sp \infty$}-functional
  calculus}, Functional analytic methods for evolution equations, Lecture Notes
  in Math., vol. 1855, Springer, Berlin, 2004, pp.~65--311.

\bibitem{mc}
A.~McIntosh, \emph{Operators which have an {$H\sb \infty$} functional
  calculus}, Miniconference on operator theory and partial differential
  equations ({N}orth {R}yde, 1986), Proc. Centre Math. Anal. Austral. Nat.
  Univ., vol.~14, Austral. Nat. Univ., Canberra, 1986, pp.~210--231.

\bibitem{nvw}
J.M.A.M.~van Neerven, M.C. Veraar, and L.W. Weis, \emph{Stochastic integration
  in {UMD B}anach spaces}, Annals Probab. \textbf{35} (2007), 1438--1478.

\bibitem{NVW11b}
\bysame, \emph{Maximal {$L^p$}-regularity for stochastic evolution equations},
  SIAM J. Math. Anal. \textbf{44} (2012), no.~3, 1372--1414. \MR{2982717}

\bibitem{NVW11}
\bysame, \emph{Stochastic maximal {$L^p$}-regularity}, Ann. Probab. \textbf{40}
  (2012), no.~2, 788--812. \MR{2952092}

\bibitem{nw}
J.M.A.M.~van Neerven and L.W. Weis, \emph{Stochastic integration of functions
  with values in a {B}anach space}, Studia Math. \textbf{166} (2005), no.~2,
  131--170.

\bibitem{ps}
J.~Pr{\"u}ss and G.~Simonett, \emph{Maximal regularity for evolution equations
  in weighted {$L\sb p$}-spaces}, Arch. Math. (Basel) \textbf{82} (2004),
  no.~5, 415--431.

\bibitem{Stein}
E.M. Stein, \emph{Harmonic analysis: real-variable methods, orthogonality, and
  oscillatory integrals}, Princeton Mathematical Series, vol.~43, Princeton
  University Press, Princeton, NJ, 1993.

\end{thebibliography}

\end{document}